\documentclass[leqno]{article}
\usepackage{amsmath}
\usepackage{amssymb}
\usepackage{dsfont}
\usepackage{multicol}
\usepackage{latexsym}
\usepackage[usenames]{color}
\usepackage[colorlinks=true,
linkcolor=webgreen,
filecolor=webbrown,
citecolor=webgreen]{hyperref}
\definecolor{webgreen}{rgb}{0,.5,0}
\definecolor{webbrown}{rgb}{.6,0,0}
\usepackage{amsthm}

\def\Re{\mathop{\mathrm{Re}}}

\def\eps{\varepsilon}

\def\N{{\mathbb{N}}}

\newtheorem{thm}{Theorem}

\begin{document}

\title{\bf Exponential unitary divisors}
\author{{\sc L\'aszl\'o T\'oth} and {\sc Nicu\c{s}or Minculete}}
\date{}
\maketitle

\begin{abstract}  We say that $d$ is an exponential unitary divisor
of $n=p_1^{a_1}\cdots p_r^{a_r}>1$ if $d=p_1^{b_1}\cdots p_r^{b_r}$, where $b_i$
is a unitary divisor of $a_i$, i.e.,  $b_i\mid a_i$ and $(b_i,a_i/b_i)=1$ for
every $i\in \{1,2,\ldots,r\}$. We survey properties of related arithmetical functions
and introduce the notion of exponential unitary perfect numbers.
\end{abstract}



\section{Introduction}

Let $n$ be a positive integer. We recall that a positive integer $d$
is called a unitary divisor of $n$ if $d\mid n$ and $(d,n/d)=1$.
Notation: $d\mid_{*} n$. If $n>1$ and has the prime factorization
$n=p_1^{a_1}\cdots p_r^{a_r}$, then $d\mid_{*} n$ iff
$d=p_1^{u_1}\cdots p_r^{u_r}$, where $u_i=0$ or $u_i=a_i$ for every
$i\in \{1,2,\ldots,r\}$. Also, $1\mid_{*} 1$.

Furthermore, $d$ is said to be an exponential divisor (e-divisor) of
$n=p_1^{a_1}\cdots p_r^{a_r}>1$ if $d=p_1^{e_1}\cdots p_r^{e_r}$, where $e_i \mid a_i$, for any
$i\in \{1,2,\ldots,r\}$. Notation: $d\mid_e n$. By convention
$1\mid_e 1$.

Let $\tau^{*}(n):=\sum_{d\mid_{*} n} 1$,
$\sigma^*(n):=\sum_{d\mid_{*} n} d$ and
$\tau^{(e)}(n):=\sum_{d\mid_{e} n} 1$, $\sigma^{(e)}(n):
=\sum_{d\mid_{e} n} d$ denote, as usual, the number and the sum of
the unitary divisors of $n$ and of the e-divisors of $n$,
respectively. These functions are multiplicative and one has
\begin{equation}
\tau^{*}(n)=2^{\omega(n)}, \quad \sigma^*(n)=(1+p_1^{a_1})\cdots (1+p_r^{a_r}),
\end{equation}
\begin{equation}
\tau^{(e)}(n)=\tau(a_1)\cdots \tau(a_r), \quad
\sigma^{(e)}(n)=\left(\sum_{d_1\mid a_1} p_1^{d_1}\right) \cdots
\left(\sum_{d_r\mid a_r} p_r^{d_r}\right),
\end{equation}
where $\omega(n):=\sum_{p\mid n} 1$ is the number of distinct prime
divisors of $n$ and $\tau(n):=\sum_{d\mid n} 1$ stands for the number
of divisors of $n$.

Note that if $n$ is squarefree, then $d\mid_* n$ iff $d\mid n$, and
$\tau^{*}(n)=\tau(n)$, $\sigma^{*}(n)= \sigma(n):=\sum_{d\mid n} d$.

Closely related to the concepts of unitary and exponential divisors are the unitary convolution and
the exponential convolution (e-convolution) of arithmetic functions
defined by
\begin{equation}
(f\times g)(n)=\sum_{d\mid_{*} n} f(d)g(n/d), \quad n\ge 1,
\label{unitary_convo}
\end{equation}
and by $(f\odot g)(1)=f(1)g(1)$,
\begin{equation}
(f\odot g)(n)=\sum_{b_1c_1=a_1} \dots \sum_{b_rc_r=a_r}
f(p_1^{b_1}\cdots p_r^{b_r}) g(p_1^{c_1}\cdots p_r^{c_r}), \quad
n>1, \label{e_convo}
\end{equation}
respectively.

The function $I(n)=1$ ($n\ge 1$) has inverses with respect to the
unitary convolution and e-convolution given by
$\mu^*(n)=(-1)^{\omega(n)}$ and $\mu^{(e)}(n)=\mu(a_1)\cdots
\mu(a_r)$, $\mu^{(e)}(1)=1$, respectively, where $\mu$ is the M\"obius function.
These are the unitary and exponential analogues of the M\"obius
function.

Unitary divisors (called block factors) and the unitary convolution
(called compounding of functions) were first considered by R.
Vaidyanathaswamy \cite{Vai1931}. The current terminology was
introduced by E. Cohen \cite{Coh1960,Coh1961}. The notions of
exponential divisor and exponential convolution were first defined
by M. V. Subbarao \cite{Sub1972}. Various properties of arithmetical
functions defined by unitary and exponential divisors, including the
functions $\tau^{*}$, $\sigma^*$, $\mu^*$, $\tau^{(e)}$,
$\sigma^{(e)}$, $\mu^{(e)}$ and properties of the convolutions
\eqref{unitary_convo} and \eqref{e_convo} were investigated by
several authors.

A positive integer $n$ is said to be unitary perfect if
$\sigma^{*}(n)=2n$. This notion was introduced by M. V. Subbarao and
L. J. Warren \cite{SubWar1966}. Until now five unitary perfect
numbers are known. These are $6=2\cdot 3$, $60=2^2\cdot 3\cdot 5$,
$90=2\cdot 3^2\cdot 5$, $87\, 360=2^6\cdot 3\cdot 5\cdot 7\cdot 13$
and the following number of $24$ digits: $146\, 361\, 946\, 186\,
458\, 562\, 560\, 000 = 2^{18}\cdot 3 \cdot 5^4 \cdot 7 \cdot 11
\cdot 13 \cdot 19 \cdot 37 \cdot 79 \cdot 109 \cdot 157 \cdot 313$.
It is conjectured that there are finitely many such numbers. It is
easy to see that there are no odd unitary perfect numbers.

An integer $n$ is called exponentially perfect (e-perfect) if
$\sigma^{(e)}(n)=2n$. This originates from M. V. Subbarao
\cite{Sub1972}. The smallest e-perfect number is $36=2^2\cdot 3^2$.
If $n$ is any squarefree number, then $\sigma^{(e)}(n)=n$, and $36n$ is e-perfect for
any such $n$ with $(n,6)=1$. Hence there are infinitely many
e-perfect numbers. Also, there are no odd e-perfect numbers, cf.
\cite{StrSub1974}. The squarefull e-perfect numbers under $10^{10}$
are: $2^2\cdot 3^2$, $2^3\cdot 3^2\cdot 5^2$, $2^2\cdot 3^3\cdot
5^2$, $2^4\cdot 3^2\cdot 11^2$, $2^4\cdot 3^3\cdot 5^2\cdot 11^2$,
$2^6\cdot 3^2 \cdot 7^2\cdot 13^2$, $2^7\cdot 3^2\cdot 5^2\cdot
7^2\cdot 13^2$, $2^6\cdot 3^3\cdot 5^2\cdot 7^2\cdot 13^2$. It is
not known if there are infinitely many squarefull e-perfect
numbers, see \cite[p.\ 110]{Guy2004}.

For a survey on results concerning unitary and exponential divisors
we refer to the books \cite{McC1986} and \cite{SanCrs2004}. See also
the papers \cite{Der2008,Hag1988,KatSub2003,KatWij1998,Pet,Sne2004,Tot2004a,
Tot2004b,Tot2007} and their references.

M. V. Subbarao \cite[Section 8]{Sub1972} says: ,,We finally remark that to every given convolution of
arithmetic functions, one can define the corresponding exponential convolution
and study the properties of arithmetical functions which arise
therefrom. For example, one can study the exponential unitary convolution, and in fact, the
exponential analogue of any Narkiewicz-type convolution, among others.''

While such convolutions were investigated by several authors, cf. \cite{HauRuo1997,Han1977}, it appears that
arithmetical functions corresponding to the exponential unitary convolution mentioned above were
not considered in the literature.

It is the aim of this paper to recover this lack. Combining the notions of e-divisors and unitary divisors
we consider in this paper exponential unitary divisors (e-unitary divisors). We review properties of
the corresponding $\tau$, $\sigma$, $\mu$ and Euler-type functions. It turns out that the asymptotic
behavior of these functions is similar to those of the functions $\tau^{(e)}$,
$\sigma^{(e)}$, $\mu^{(e)}$ and $\phi^{(e)}$ (the latter one will be given in Section 3). We define
the e-unitary perfect numbers, not considered before, and state some open problems.


\section{Exponential unitary divisors}

We say that $d$ is an exponential unitary divisor (e-unitary
divisor) of $n=p_1^{a_1}\cdots p_r^{a_r}>1$ if $d=p_1^{b_1}\cdots
p_r^{b_r}$, where $b_i \mid_{*} a_i$, for any $i\in
\{1,2,\ldots,r\}$. Notation: $d\mid_{e*} n$. By convention
$1\mid_{e*} 1$.

For example, the e-unitary divisors of $n=p^{12}$, with $p$ prime, are
$d=p, p^3, p^4,p^{12}$, while its e-divisors are $d=p,p^2,p^3,p^4,p^6,p^{12}$.

Let $\tau^{(e)*}(n):=\sum_{d\mid_{e*} n} 1$ and
$\sigma^{(e)*}(n):=\sum_{d\mid_{e*} n} d$ denote the number and the
sum of the e-unitary divisors of $n$, respectively. It is immediate
that these functions are multiplicative and we have
\begin{equation}
\tau^{(e)*}(n)=\tau^{*}(a_1)\cdots \tau^{*}(a_r)=
2^{\omega(a_1)+\ldots +\omega(a_r)}, \quad \sigma^{(e)*}(n)=
\left(\sum_{d_1\mid_{*} a_1} p_1^{d_1}\right) \cdots
\left(\sum_{d_r\mid_{*} a_r} p_r^{d_r}\right).
\end{equation}

If $n$ is e-squarefree, i.e., $n=1$ or $n>1$ and all the exponents
in the prime factorization of $n$ are squarefree, then $d\mid_{e*}
n$ iff $d\mid_{e} n$, and $\tau^{(e)*}(n)=\tau^{(e)}(n)$,
$\sigma^{(e)*}(n)= \sigma^{(e)}(n)$.

Note that for any $n>1$ the values $\tau^{(e)*}(n)$ and $\sigma^{(e)*}(n)$ are even.

The corresponding exponential unitary convolution (e-unitary convolution) is given by
\\ $(f\odot_* g)(1)=f(1)g(1)$,
\begin{equation}
(f\odot_* g)(n)=\sum_{\substack{b_1c_1=a_1\\(b_1,c_1)=1}} \dots
\sum_{\substack{b_rc_r=a_r\\(b_r,c_r)=1}} f(p_1^{b_1}\cdots
p_r^{b_r}) g(p_1^{c_1}\cdots p_r^{c_r}), \label{e_unitary_convo}
\end{equation}
with the notation $n=p_1^{a_1}\cdots p_r^{a_r}>1$.

The arithmetical functions form a commutative semigroup under
\eqref{e_unitary_convo} with identity $\mu^2$. A function $f$ has an
inverse with respect to the e-unitary convolution iff $f(1)\ne 0$
and $f(p_1\cdots p_k)\ne 0$ for any distinct primes
$p_1,\ldots,p_k$.

The inverse of the function $I(n)=1$ ($n\ge 1$) with respect to the
e-unitary convolution is the function
$\mu^{(e)*}(n)=\mu^*(a_1)\cdots \mu^*(a_r)=(-1)^{\omega(a_1)+\ldots
+\omega(a_r)}$, $\mu^{(e)*}(1)=1$.

These properties of convolution \eqref{e_unitary_convo} are special cases of those of a more general
convolution, involving regular convolutions of Narkiewicz-type,
mentioned in the Introduction.

Remark. It is possible to define ,,unitary exponential divisors'' (in the reverse order) in the following way.
An integer $d$ is a unitary exponential divisor (unitary e-divisor) of $n=p_1^{a_1}\cdots p_r^{a_r}>1$ if
$d\mid n$ and the integers $d$ and $n/d$ are exponentially coprime. This means that, denoting
$d=p_1^{b_1}\cdots p_r^{b_r}$, we require $d$ and $n/d$ to have the same prime factors as $n$,
i.e., $1\le b_i< a_i$, and $(b_i,a_i-b_i)=1$ for any $i\in \{1,2,\ldots,r\}$. This is fulfilled
iff $n$ is squarefull, i.e., $a_i\ge 2$ and $(b_i,a_i)=1$ for every $i\in \{1,2,\ldots,r\}$.
Hence the number of unitary e-divisors of $n>1$ is $\phi(a_1)\cdots \phi(a_r)$ ($\phi$ is Euler's function)
or $0$, according as $n$ is squarefull or not. We do not go here into other details. For exponentially
coprime integers cf. \cite{Tot2004a}.


\section{Arithmetical functions defined by exponential unitary divisors}

As noted before, the functions $\tau^{(e)*}$ and $\sigma^{(e)*}$ are
multiplicative. Also, for any prime $p$, $\tau^{(e)*}(p)=1$,
$\tau^{(e)*}(p^2)=2$, $\tau^{(e)*}(p^3)=2$, $\tau^{(e)*}(p^4)=2$,
$\tau^{(e)*}(p^5)=2$, ..., $\sigma^{(e)*}(p)=p$,
$\sigma^{(e)*}(p^2)=p+p^2$, $\sigma^{(e)*}(p^3)=p+p^3$,
$\sigma^{(e)*}(p^4)=p+p^4$, $\sigma^{(e)*}(p^5)=p+p^5$, ....
Observe that the first difference compared with the functions
$\tau^{(e)}$ and $\sigma^{(e)}$ occurs for $p^4$ (which is not
e-squarefree).

The function $\tau^{(e)*}(n)$ is identic with the function
$t^{(e)}(n)$, defined as the number of e-squarefree e-divisors of $n$
and investigated by L. T\'oth \cite{Tot2007}. According to
\cite[Th.\ 4]{Tot2007},
\begin{equation}
\sum_{n\le x} \tau^{(e)*}(n)= C_1x +C_2 x^{1/2} + {\cal
O}(x^{1/4+\eps}), \label{tau_e_unit_sum}
\end{equation}
for every $\eps >0$, where $C_1, C_2$ are constants given by
\begin{equation}
C_1:=\prod_p \left(1 + \frac1{p^2} + \sum_{a=6}^{\infty}
\frac{2^{\omega(a)}-2^{\omega(a-1)}}{p^a}\right),
\end{equation}
\begin{equation}
C_2:=\zeta(1/2) \prod_p \left(1 + \sum_{a=4}^{\infty}
\frac{2^{\omega(a)}-2^{\omega(a-1)}-2^{\omega(a-2)}+
2^{\omega(a-3)}}{p^{a/2}}\right).
\end{equation}

The error term of \eqref{tau_e_unit_sum} was improved into ${\cal
O}(x^{1/4})$ by Y.-F. S. P\'etermann \cite[Th.\ 1]{Pet} showing that
\begin{equation}
\sum_{n=1}^{\infty} \frac{t^{(e)}(n)}{n^s}
=\frac{\zeta(s)\zeta(2s)}{\zeta(4s)} H(s), \quad \Re s>1,
\end{equation}
where $H(s)=\sum_{n=1}^{\infty} \frac{h(n)}{n^s}$ is absolutely
convergent for $\Re s >1/6$.

For the maximal order of the function $\tau^{(e)*}$ we have
\begin{equation}
\limsup_{n\to \infty} \frac{\log \tau^{(e)*}(n) \log \log n}{\log n}
= \frac1{2} \log 2, \label{limsup_tau_e_unit}
\end{equation}
this is proved (for $t^{(e)}(n)$) in \cite[Th.\ 5]{Tot2007}.
\eqref{limsup_tau_e_unit} holds also for the function $\tau^{(e)}$
instead of $\tau^{(e)*}$, cf. \cite{Sub1972}.

For the maximal order of the function $\sigma^{(e)*}$ we have

\begin{thm}
\begin{equation}
\limsup_{n\to \infty} \frac{\sigma^{(e)*}(n)}{n\log \log
n}=\frac{6}{\pi^2}e^{\gamma}, \label{limsup_sigma_e_unit}
\end{equation}
where $\gamma$ is Euler's constant.
\end{thm}

\begin{proof} This is a direct consequence of the following general result of L. T\'oth and E. Wirsing
\cite[Cor.\ 1]{TotWir2003}: Let $f$ be a nonnegative real-valued
multiplicative function. Suppose that for all primes $p$ we have
$\varrho(p):=\sup_{\nu \ge 0}f(p^\nu) \le (1-1/p)^{-1}$ and that for
all primes $p$ there is an exponent $e_p=p^{o(1)}$ such that
$f(p^{e_p})\ge 1+1/p$. Then
\begin{equation}
\limsup_{n\to \infty} \frac{f(n)}{\log \log n}=e^{\gamma}\prod_p
\left(1-\frac1{p}\right) \varrho(p).
\end{equation}
Apply this for $f(n)=\sigma^{(e)*}(n)/n$. Here $f(p)=1$, $f(p^2)=1+1/p$ and for $a\ge 2$,
$f(p^a)\le \sigma^{(e)}(p^a)/p^a\le 1+1/p$. Hence $\varrho(p)=1+1/p$ and
we can choose $e_p=2$ for all $p$.
\end{proof}

\eqref{limsup_sigma_e_unit} holds also for the function
$\sigma^{(e)}$ instead of $\sigma^{(e)*}$. For the function
$\mu^{(e)*}$ one has:

\begin{thm} (i) The Dirichlet series of $\mu^{(e)*}$
is of form
\begin{equation}
\sum_{n=1}^{\infty} \frac{\mu^{(e)*}(n)}{n^s} =
\frac{\zeta(s)}{\zeta^2(2s)} W(s), \quad \Re s > 1,
\end{equation}
where $W(s):=\sum_{n=1}^{\infty} \frac{w(n)}{n^s}$ is absolutely
convergent for $\Re s > 1/4$.

(ii)
\begin{equation} \sum_{n\le x} \mu^{(e)*}(n)=C_3x+ {\cal
O}(x^{1/2}\exp(-c (\log x)^{\Delta}),
\end{equation} where
\begin{equation}
C_3:= \prod_p \left(1+\sum_{a=2}^{\infty}
\frac{(-1)^{\omega(a)}-(-1)^{\omega(a-1)}}{p^a}\right),
\end{equation}
and $\Delta<9/25= 0.36$ and $c>0$ are constants. \end{thm}

\begin{proof} A similar result was proved for the function $\mu^{(e)}$ in \cite[Th.\
2]{Tot2007} (with the auxiliary Dirichlet series absolutely convergent for $\Re s > 1/5$). 
The same proof works out in case of $\mu^{(e)*}$. The error term can be improved assuming the
Riemann hypothesis, cf. \cite{Tot2007}.
\end{proof}

The unitary analogue of Euler's arithmetical function, denoted by
$\phi^*$ is defined as follows. Let $(k,n)_*:=\max \{d\in \N: d\mid
k, d\mid_* n\}$ and let
\begin{equation} \label{phi*}
\phi^*(n):= \# \{k\in \N: 1\le k\le n, (k,n)_*=1\},
\end{equation}
which is multiplicative and $\phi^*(p^a)=p^a-1$ for every prime
power $p^a$ ($a\ge 1$). Why do we not consider here the greatest
common unitary divisor of $k$ and $n$? Because if we do so
the resulting function is not multiplicative and its properties are
not so close to those of Euler's function $\phi$, cf. \cite{Tot2009}.

Furthermore, for $n=p_1^{a_1}\cdots p_r^{a_r}>1$ let $\phi^{(e)}(n)$
denote the number of divisors $d$ of $n$ such that $d$ and $n$ are
exponentially coprime, i.e., $d=p_1^{b_1}\cdots p_r^{b_r}$, where
$1\le b_i\le a_i$ and $(b_i,a_i)=1$ for any $i\in \{1,\ldots,r\}$.
By convention, let $\phi^{(e)}(1)=1$. This is the exponential
analogue of the Euler function, cf. \cite{Tot2004b}. Here $\phi^{(e)}$ is multiplicative
and
\begin{equation}
\phi^{(e)}(n)=\phi(a_1)\cdots \phi(a_r), \quad n>1.
\end{equation}

We define the e-unitary Euler function in this way: for
$n=p_1^{a_1}\cdots p_r^{a_r}>1$ let $\phi^{(e)*}(n)$ denote the
number of divisors $d$ of $n$ such that $d=p_1^{b_1}\cdots
p_r^{b_r}$, where $1\le b_i\le a_i$ and $(b_i,a_i)_*=1$ for any
$i\in \{1,\ldots,r\}$. By convention, let $\phi^{(e)*}(1)=1$. Then
$\phi^{(e)*}$ is multiplicative and
\begin{equation}
\phi^{(e)*}(n)=\phi^*(a_1)\cdots \phi^*(a_r), \quad n>1.
\end{equation}

\begin{thm}
\begin{equation}
\sum_{n\le x} \phi^{(e)*}(n)=C_4x +C_5 x^{1/3} + {\cal
O}(x^{1/4+\eps}),
\end{equation}
for every $\eps >0$, where $C_4, C_5$ are constants given by
\begin{equation}
C_4:=\prod_p \left(1 + \sum_{a=3}^{\infty}
\frac{\phi^*(a)-\phi^*(a-1)}{p^a}\right),
\end{equation}
\begin{equation}
C_5:=\zeta(1/3) \prod_p \left(1 + \frac1{p^{4/3}}+
\sum_{a=5}^{\infty} \frac{\phi^*(a)-\phi^*(a-1)-\phi^*(a-3)+
\phi^*(a-4)}{p^{a/3}}\right).
\end{equation}
\end{thm}

\begin{proof} A similar result was proved for the function $\phi^{(e)}$ in \cite[Th.\ 1]{Tot2004b},
with error term ${\cal O}(x^{1/5+\varepsilon})$, improved into
${\cal O}(x^{1/5}\log x)$ by Y.-F. S. P\'etermann \cite[Th.\ 1]{Pet}. The same proof works out
in case of $\phi^{(e)*}$.
\end{proof}

\begin{thm}
\begin{equation}
\limsup_{n\to \infty} \frac{\log \phi^{(e)*}(n) \log \log n}{\log
n}= \frac{\log 4}{5}. \label{limsupphi}
\end{equation}
\end{thm}

\begin{proof}  We apply the following general
result given in \cite{SurSit1975}: Let $F$ be a multiplicative
function with $F(p^a)=f(a)$ for every prime power $p^a$, where $f$
is positive and satisfying $f(n)={\cal O}(n^\beta)$ for some fixed
$\beta >0$. Then
\begin{equation}
\limsup_{n\to \infty} \frac{\log F(n) \log \log n}{\log n} =
\sup_{m\ge 1} \frac{\log f(m)}{m}.
\end{equation}

Let $F(n)=\phi^{(e)*}(n)$, $f(a)=\phi^*(a)$, $L(m)=(\log f(m))/m$. Here $L(1)=L(2)=0$,
$L(3)=(\log 2)/3 \approx0.231$, $L(4) =(\log 3)/4\approx 0.274$, $L(5)=(\log 4)/5\approx
0.277$, $L(6)=(\log 5)/6\approx 0.268$, $L(7) =(\log 6)/7\approx 0.255$, and $L(m)\le (\log
m)/m\le (\log 8)/8\approx 0.259$ for $m\ge 8$, using that $(\log
m)/m$ is decreasing. This proves the result.
\end{proof}

\eqref{limsupphi} holds also for the function $\phi^{(e)}$ instead of $\phi^{(e)*}$, cf. \cite{Tot2004b}.

These results show that the asymptotic behavior of the functions
$\tau^{(e)*}$, $\sigma^{(e)*}$, $\mu^{(e)*}$ and $\phi^{(e)*}$ is very close to
those of the functions $\tau^{(e)}$, $\sigma^{(e)}$, $\mu^{(e)}$ and $\phi^{(e)}$.

This is confirmed also by the next result.

\begin{thm}
\begin{equation}
\sum_{n\le x} \frac{\tau^{(e)*}(n)}{\tau^{(e)}(n)}= x \prod_p
\left(1+\sum_{a=4}^{\infty} \frac{2^{\omega(a)}/\tau(a)-2^{\omega(a-1)}/\tau(a-1)}{p^a} \right)+
{\cal O}\left(x^{1/4}\log x\right).
\label{taue*/taue}
\end{equation}

A similar asymptotic formula, with the same error term, is valid also for the quotients
\\ $\sigma^{(e)*}(n)/\sigma^{(e)}(n)$ and $\phi^{(e)}(n)/\phi^{(e)*}(n)$ (in the reverse order for the last one).
\end{thm}

\begin{proof} This follows from the following general result, which may be known. Let $g$ be a complex valued
multiplicative function such that $|g(n)|\le 1$ for every $n\ge 1$
and $g(p)=g(p^2)=g(p^3)=1$ for every prime $p$. Then
\begin{equation}
\sum_{n\le x} g(n)= x \prod_p \left(1+\sum_{a=4}^{\infty} \frac{g(p^a)-g(p^{a-1})}{p^a} \right)+
{\cal O}\left(x^{1/4}\log x\right).
\label{h}
\end{equation}

To obtain \eqref{h}, which is similar to \cite[Th.\ 1]{Tot2007}, let $h=g*\mu$ in terms of the
Dirichlet convolution. Then $h$ is multiplicative, $h(p)=h(p^2)=h(p^3)=0$,
$h(p^a)=g(p^a)-g(p^{a-1})$ and $|h(p^a)|\le 2$ for every prime $p$ and every $a\ge 4$. Hence $|h(n)|\le
\ell_4(n) 2^{\omega(n)}$ for every $n\ge 1$, where $\ell_4(n)$ stands for the
characteristic function of the $4$-full integers. Note that
\begin{equation}
\ell_4(n)2^{\omega(n)}=\sum_{d^4e=n} \tau(d)v(e),
\end{equation}
where the function $v$ is given by
\begin{equation}
\sum_{n=1}^{\infty} \frac{v(n)}{n^s}=\prod_p
\left(1+\frac{2}{p^{5s}}+\frac{2}{p^{6s}}+ \frac{2}{p^{7s}} -\frac1{p^{8s}}-\frac{2}{p^{9s}}-
\frac{2}{p^{10s}} -\frac{2}{p^{11s}}\right),
\end{equation}
absolutely convergent for $\Re s>1/5$. We obtain \eqref{h} by usual estimates, cf. the proof of
\cite[Th.\ 1]{Tot2007}. \end{proof}

Note also, that $\mu^{(e)}(n)/\mu^{(e)*}(n)=|\mu^{(e)}(n)|$ is the characteristic function of the e-squarefree
integers $n$. Asymptotic formulae for $|\mu^{(e)}(n)|$ were given in \cite[Th.\ 2]{Wu1995},
\cite[Th.\ 3]{Tot2007}.


\section{Exponential unitary perfect numbers}

We call an integer $n$ exponential unitary perfect (e-unitary perfect) if $\sigma^{(e)*}(n)=2n$.

If $n$ is e-squarefree, then $n$ is e-unitary perfect iff $n$ is
e-perfect. Consider the squarefull e-unitary perfect numbers. The
first three such numbers given in the Introduction, that is
$36=2^2\cdot 3^2$, $1\,800= 2^3\cdot 3^2\cdot 5^2$ and
$2\,700=2^2\cdot 3^3\cdot 5^2$ are e-squarefree, therefore also
e-unitary perfect. It follows that there are infinitely many
e-unitary perfect numbers.

The smallest number which is e-perfect but not e-unitary perfect is $17\,424 = 2^4\cdot 3^2 \cdot 11^2$.

\begin{thm} There are no odd e-unitary perfect numbers.
\end{thm}

\begin{proof} Let $n=p_1^{a_1}\cdots p_r^{a_r}$ be an odd e-unitary
perfect number. That is \begin{equation}
\sigma^{(e)*}(p_1^{a_1})\cdots \sigma^{(e)*}(p_r^{a_r})=
2p_1^{a_1}\cdots p_r^{a_r}. \label{proof_1}
\end{equation}

We can assume that $a_1,\ldots,a_r\ge 2$, i.e. $n$ is squarefull.
(if $a_i=1$ for an $i$, then $\sigma^{(e)*}(p_i)=p_i$ and we can
simplify in \eqref{proof_1} by $p_i$).

Now each $\sigma^{(e)*}(p_i^{a_i})= \sum_{d\mid_* a_i} p_i^d$ is
even, since the number of terms is $2^{\omega(a_i)}$, which is even.

From \eqref{proof_1} we obtain that $r=1$ and have
\begin{equation}
\sigma^{(e)*}(p_1^{a_1})= 2p_1^{a_1}. \label{proof_2}
\end{equation}

Using that $a_1\ge 2$,
\begin{equation}
2= \frac{\sigma^{(e)*}(p_1^{a_1})}{p_1^{a_1}}\le
\frac{\sigma^{(e)}(p_1^{a_1})}{p_1^{a_1}}\le 1+\frac1{p_1}\le
1+\frac1{3}<2
\end{equation}
is a contradiction, and the proof is complete.
\end{proof}

We state the following open problems.

Problem 1. Is there any e-unitary perfect number which is not e-squarefree,
therefore not e-perfect?

Problem 2. Is there any e-unitary perfect number which is not divisible by $3$?


\vskip1mm {\sc L\'aszl\'o T\'oth}, University of P\'ecs, Department of
Mathematics, Ifj\'us\'ag u. 6, 7624 P\'ecs, Hungary, E-mail: ltoth@gamma.ttk.pte.hu

\vskip1mm {\sc Nicu\c{s}or Minculete}, ,,Dimitrie Cantemir'' University of Bra\c{s}ov, Romania,
\\ E-mail: minculeten@yahoo.com

\end{document}